\documentclass[aap]{imsart}

\RequirePackage{amsthm,amsmath,amsfonts,amssymb}
\RequirePackage[numbers]{natbib}
\RequirePackage[colorlinks,citecolor=blue,urlcolor=blue]{hyperref}
\RequirePackage{graphicx}

\startlocaldefs
\theoremstyle{plain}
\newcommand\numberthis{\addtocounter{equation}{1}\tag{\theequation}}

\def\R{\mathbb R}
\def\Z{\mathbb Z}
\def\eqd{\buildrel\hbox{\footnotesize d}\over =}
\def\tod{\buildrel\hbox{\footnotesize d}\over \to}
\newtheorem{thm}{Theorem}[section]
\newtheorem{lem}{Lemma}[section]
\newtheorem{cor}{Corollary}[section]

\endlocaldefs

\begin{document}

\begin{frontmatter}
\title{The Berry-Esseen Theorem for Circular $\beta$-ensemble}
\runtitle{Berry-Esseen Theorem}

\begin{aug}
\author[A]{\fnms{Renjie}~\snm{Feng}\ead[label=e1]{feng@mpim-bonn.mpg.de}},
\author[B]{\fnms{Gang}~\snm{Tian}\ead[label=e2]{gtian@math.pku.edu.cn}}
\and
\author[B]{\fnms{Dongyi}~\snm{Wei}\ead[label=e3]{jnwdyi@pku.edu.cn}}
\address[A]{Max Planck Institute for Mathematics, Bonn, Germany, 53111.\printead[presep={,\ }]{e1}}

\address[B]{Beijing International Center for Mathematical Research and School of Mathematical Sciences, Peking University, Beijing, China, 100871.\printead[presep={,\ }]{e2,e3}}
 
\end{aug}

\begin{abstract}
We will prove the Berry-Esseen theorem for the number counting function of the circular $\beta$-ensemble (C$\beta$E), which will imply the   central limit theorem
for the number of points in arcs of the unit circle in mesoscopic and macroscopic scales.
We will prove the main result by estimating the characteristic functions of the Pr\"ufer phases and the number counting function,  which will imply the
  uniform upper and lower bounds of their variance.  We also show that the similar results hold for the Sine$_\beta$ process.  As a direct application of  the uniform variance bound, we can prove the normality of the linear statistics when the test function $f(\theta)\in W^{1,p}(S^1)$ for some $p\in(1,+\infty)$.\end{abstract}



\end{frontmatter}

\section{Introduction}
The circular $ \beta$-ensemble (measure $ \mu_{\beta,n}$, $\beta>0$) is a random process on the unit circle and the joint density of its eigenangles $ \theta_j\in[0,2\pi)$, $1\leq j\leq n$ with respect to the Lebesgue measure is\begin{align*}&J(\theta_1,\cdots, \theta_n)=\frac{1}{C_{\beta,n}}\prod_{j<k}|e^{i\theta_j}-e^{i\theta_k}|^{\beta},
\end{align*} where $\beta>0$ and $C_{\beta,n}=(2\pi)^{n}\frac{\Gamma(1+\beta n/2)}{(\Gamma(1+\beta /2))^n}$ is the normalization constant \cite{For}.

There are many results regarding the normality of C$\beta$E and G$\beta$E (we refer to \cite{For} for the definition of G$\beta$E).
 For C$\beta$E, Killip \cite{Ki08} proved the central limit theorem for the
number of points in the fixed arcs, and the variance is logarithmic in $n$, where the result can be considered as the macroscopic statistics.  For G$\beta$E, Costin-Lebowitz proved the normality of eigenvalues in the particular cases $\beta\in\{1, 2, 4\}$ and the variance is also logarithmic with respect to the
mean \cite{CL}. These results can be extended  to more general point processes, 
we refer to \cite{Bgm, BLS, Bl, DS, JM,  J1, J, JG, L1, L3, L2, So1, So3, So2, S} and the references therein.

Recently, in \cite{NV},  Najnudel-Vir\'ag proved uniform upper bounds on the variance of
the number of points in intervals for both C$\beta$E and G$\beta$E.
Their bounds are uniform in $n$ which cover microscopic, mesoscopic and macroscopic scales. And if one rescales the interval or the arc in such a way that the average
spacing between the points has order 1, then the upper bounds are logarithmic in the length of the
interval or the arc.  To be more precise, in the case of C$\beta$E, let's write $N_n(a,b)$ for the
number of points in a sample from $ \mu_{\beta,n}$ that lie in the arc between $a$ and $b$,  Najnudel-Vir\'ag proved  the following uniform upper bound
\begin{equation}\label{old}
\mathbb{E}[ \left| N_n(0,\theta)-{n\theta}/{(2\pi)}\right|^2]\leq C_{\beta}\ln(2+n\theta).
\end{equation}

In this paper, we will study the normality of the number counting function and the linear statistics for C$\beta$E.  Our first main result is the following Berry-Esseen theorem \cite{B, E} for the number counting function, which is novel and not proved elsewhere.\begin{thm}\label{prop2}Let $ \theta\in(0,\pi]$ that may depend on $n$, we have the uniform estimate  \begin{align*}
&\sup_{x\in\R}\left|\mathbb{P}\left[\sqrt{\frac{\pi^2\beta}{2\ln(2+n\theta)}}\left[N_n(0,\theta)-\frac{n\theta}{2\pi}\right]\leq x\right]-\int_{-\infty}^x\frac{e^{-t^2/2}}{\sqrt{2\pi}}dt\right|\leq
 \frac{C}{(\ln(2+n\theta))^{\frac{1}{2}}},
\end{align*}here $C>0$ is a constant depending only on $ \beta$.\end{thm}As a direct consequence of Theorem \ref{prop2}, we have  the following central limit theorem for the number of points in arcs of the unit circle for C$\beta$E in both mesoscopic and macroscopic scales.\begin{cor}\label{cor2}Let $\theta_n\in(0,\pi]$ and  assume $n\theta_n\to+\infty$,  then we have the following central limit theorem \begin{align*}
&\sqrt{\frac{\pi^2\beta}{2\ln(2+n\theta_n)}}[N_n(0,\theta_n)-\frac{n\theta_n}{2\pi}]\tod N(0,1)
\end{align*} as $n\to\infty$. Here, $N(0,1)$ denotes the standard Gaussian random variable, and $\tod$ means the convergence in distribution.  \end{cor}Notice that, Corollary \ref{cor2} is the main result proved  in \cite{Ki08} for the case where $\theta_n=\theta $ is fixed.  

 To show the key steps to prove Theorem \ref{prop2}, we begin with some preliminary results proved in \cite{Ki08, KS09}. 
  Let $ \gamma_j\sim\Theta_{\beta(j+1)+1}$ be independent random variables  for $ j\geq 0$ and let $ \eta$ be a
uniform random variable on $[0,2\pi)$ independent of $ (\gamma_j)_{j\geq 0}$. We define the so-called Pr\"ufer phases  $(\psi_k(\theta))_{\theta\in\R,k\geq 0}$ as follows: $\psi_0(\theta)=\theta$ and for $ k\geq 0,$\begin{align*}
&\psi_{k+1}(\theta)=\psi_{k}(\theta)+\theta+2{\rm Im}\ln\left(\frac{1-\gamma_k}
{1-\gamma_ke^{i\psi_{k}(\theta)}}\right).
\end{align*}Then the random set\begin{align*}
&\{\theta\in\R,\psi_{n-1}(\theta)\equiv\eta({\rm mod}\ 2\pi)\}
\end{align*}has the same law as the set of all determinations of the arguments of the $n$ points of a C$\beta$E (see \S 2.2 in \cite{Ki08} for more details of this result). Here, a complex random variable $X$ with values in the unit disk $ \mathbb{D}$ is $ \Theta_{\nu}$-distributed (for $\nu > 1$) if (see Definition 2.1 in \cite{Ki08} also)\begin{align*}
&\mathbb{E}[f(X)]=\frac{\nu-1}{2\pi}\iint_{\mathbb{D}}f(z)(1-|z|^2)^{(\nu-3)/2}\mathrm{d}^2z.
\end{align*} Simple computations show $$\mathbb{E}[X]=0,\,\, \mathbb{E}[|X|^2]=\frac{2}{\nu+1},\,\, \mathbb{E}[|X|^4]=\frac{8}{(\nu+1)(\nu+3)} $$ and\begin{align}\label{EXm}
&\mathbb{E}[(-\ln(1-|X|^2))^m]=\frac{\nu-1}{2}\int_{0}^{\infty}t^me^{-(\nu-1)t/2}\mathrm{d}t=
\Gamma(m+1)\left(\frac{2}{\nu-1}\right)^m,\end{align}where we change the variable $e^{-t}=1-|X|^2$.

The above result tells us
that $$N_n(0,\theta)\eqd\lfloor\frac{\psi_{n-1}(\theta)-\eta}{2\pi}\rfloor+1$$ for $\theta\in(0,2\pi).$ Here $\lfloor x \rfloor$ is the floor function, $X\eqd Y$ means that the random variables $X,Y$ have the same distribution, we also used the fact that $\psi_{n-1}(0)=0, $ and that $\psi_{n-1}(\theta)$ is increasing with respect to $ \theta$. By rotational invariance we have $$\mathbb{E}[N_n(0,\theta)]={n\theta}/{(2\pi)} $$ and $$N_n(0,2\pi-\theta)\eqd N_n(\theta,2\pi)=n-N_n(0,\theta),$$ i.e.,  there is a natural symmetry between $ \theta$ and $2\pi-\theta$, therefore, it is enough to study the case $\theta\in(0,\pi]. $

Throughout the article,
we will use $C>0$ to denote a universal constant depending only on $\beta$ which may change
from line to line.

To prove Theorem \ref{prop2}, the key lemma  is the following estimate regarding the characteristic function of the Pr\"ufer phases.
\begin{lem}\label{lem4}Let $ \theta\in(0,\pi],\ \lambda\in\R,\ \lambda^2\leq\beta/8$. There exists a constant $C>0$  depending only on $ \beta$ such that\begin{align*}
&|\mathbb{E}[e^{i\lambda (\psi_{n-1}(\theta)-n\theta)}]-e^{-(4\lambda^2/{\beta})\ln(2+n\theta)}|\leq
 C\lambda^2e^{-(4\lambda^2/{\beta})\ln(2+n\theta)}.
\end{align*}\end{lem}


Moreover, as a consequence of Lemma \ref{lem4}, we can prove the
following uniform bound first for the variance of $\psi_{n-1}(\theta) $, then for the variance of $N_n(0,\theta)$.
\begin{cor}\label{cor1}There exists a constant $C>0$  depending only on $ \beta$ such that for $\theta\in(0,\pi],\ n\in\Z,\ n>0,$ $ \lambda\in[-2\pi\sqrt{\beta/8},2\pi\sqrt{\beta/8}],$ we have \begin{equation}\label{ds}|\mathbb{E}[(\psi_{n-1}(\theta)-n\theta)^2]-(8/{\beta})\ln(2+n\theta)|\leq C,\end{equation}\begin{equation}\label{uds}|\mathbb{E}[e^{i\lambda ( N_n(0,\theta)-n\theta/(2\pi))}]-e^{-\lambda^2/({\beta}\pi^2)\cdot\ln(2+n\theta)}|\leq C\lambda^2,\end{equation}
\begin{equation}\label{udd}\left|\mathbb{E}\left[ \left| N_n(0,\theta)-\frac{n\theta}{2\pi}\right|^2\right]-\frac{2\ln(2+n\theta)}{\pi^2\beta}\right|\leq C.
\end{equation}\end{cor} The inequality \eqref{udd} gives both the upper and lower uniform variance bounds which improves the estimate \eqref{old}. 

The Sine$_\beta$ point process is the scaling limit of the C$\beta$E, and its central limit theorem has been proved in \cite{KVV}. In this article, we can further prove the
following uniform variance bound and  the  Berry-Esseen theorem for the Sine$_\beta$ point process. Let's denote $\text{Card}(A)$ the 
cardinality of  a set $A$. \begin{cor}\label{cor3}Let $L$ be the $\operatorname{Sine}_\beta$ point process, there exists a constant $C>0$  depending only on $ \beta$ such that for $x>0,$ we have \begin{equation}\label{L1}|\mathbb{E}[(\operatorname{Card}(L\cap[0,x])-x/(2\pi))^2]-2/({\beta}\pi^2)\cdot\ln(2+x)|\leq C,\end{equation}\begin{align}\label{L2}&\sup_{y\in\R}\left|\mathbb{P}\left[\sqrt{\frac{\pi^2\beta}{2\ln(2+x)}}\left[\operatorname{Card}(L\cap[0,x])-\frac{x}{2\pi}\right]\leq y\right]-\int_{-\infty}^y\frac{e^{-t^2/2}}{\sqrt{2\pi}}dt\right|\\ \nonumber&\leq
{C}{(\ln(2+x))^{-\frac{1}{2}}}.\end{align}
\end{cor}

In the end, as a direct application of the uniform variance bound \eqref{udd}, we can prove the normality of the linear statistics for C$\beta$E when the test function is in $W^{1,p}(S^1)$ for some $p\in(1,+\infty),$ and $p$ will be fixed. Let's denote
$$\xi_n=\sum_{j=1}^n \delta_{\theta_j}$$
the empirical measure of a sample from $ \mu_{\beta,n}$, and we consider the  linear statistics $$ \langle\xi_n,f\rangle=\sum_{j=1}^nf(\theta_j).$$
We will prove the following result.
\begin{thm}\label{lem19} Let $f\in W^{1,p}(S^1)$ be real valued and periodic function with $f(0)=f(2\pi)$,  and $\int_0^{2\pi}f(x)dx=0$ where $p\in(1,+\infty)$, then  $\langle\xi_n,f\rangle $ converges in law to a Gaussian random variable of mean zero and variance $ 2\sigma^2,$ where\begin{align*}
\sigma^2&=\frac{2}{\beta}\sum_{j=1}^{+\infty}j|a_j|^2,\,\,a_j=\frac{1}{2\pi}\int_0^{2\pi}f(x)e^{-ijx}dx,\ j\in\Z.
\end{align*}\end{thm}
To prove Theorem \ref{lem19}, we will need the variance estimate of the linear statistics (see Lemma \ref{lem18} in \S\ref{linear}) which is based on the uniform variance bound \eqref{udd}. The rest  proof makes use of Lemma \ref{lem21} (proved in
\cite{JM}) and the approximation of the $W^{1,p}(S^1)$ space by the F\'ejer kernel.


Such result was first proved for CUE with $\beta=2$ by Diaconis-Evans in \cite{DE2} for the test function $f\in H^{1/2}(S^1)$. 
In  \cite{JM},  Jiang-Matsumoto proved the result for C$\beta$E if
$f(x)$ is a polynomial of $e^{ix}$, and they also proved the cases $ \beta=1$ if $f\in H^{1/2}(S^1)$ and $ \beta=4$ if $f\in H^{1/2+}(S^1)$. 
We also refer to \cite{DS, J1, J, L2, W} for other related results.

The organization of the article is as follows. In \S\ref{bac}, we will review some known results on C$\beta$E which are proved in \cite{Ki08, KS09}. In \S\ref{the},
we will derive Lemma \ref{lem4} and prove Corollary \ref{cor1}. In \S\ref{theorem}, we will finish the proof of Theorem \ref{prop2}.
In \S\ref{cool}, we will prove Corollary \ref{cor3} for the Sine${_\beta}$ process. In \S\ref{linear},  as an application of the uniform variance bound \eqref{udd}, we will prove Theorem \ref{lem19}.

\textbf{Acknowledgement:} We are indebted to the anonymous reviewers for providing insightful comments, this paper would not have been possible without their supportive work. 


\section{Preliminary results}\label{bac}
In this section, we will collect several properties regarding C$\beta$E  proved in  \cite{Ki08,KS09} which will be useful in the proof of Theorem \ref{prop2}.

Now we introduce\begin{align*}
&\Upsilon(\psi,\alpha)=-2{\rm Im}\ln[1-\alpha e^{i\psi}]={\rm Im}\sum_{l=1}^{+\infty}\frac{2}{l}e^{il\psi}\alpha^l ,
\end{align*}and $$\Upsilon_1(\psi,\alpha)=\Upsilon(\psi,\alpha)-\Upsilon(0,\alpha).$$ Then we have \cite{KS09}\begin{align*}
&\psi_{k+1}(\theta)=\psi_{k}(\theta)+\theta+\Upsilon_1(\psi_{k}(\theta),\gamma_k).
\end{align*}We have the following estimates about $\Upsilon $ (Lemma 2.5 in \cite{Ki08}).\begin{lem}\label{lem2}Suppose $ \phi,\psi\in\R$ and $\alpha\sim\Theta_{\nu}.$ If $ \widetilde{\Upsilon}(\psi,\alpha)=2{\rm Im}[\alpha e^{i\psi}]$, then \begin{align*}
&\mathbb{E}[{\Upsilon}(\psi,\alpha)]=\mathbb{E}[\widetilde{\Upsilon}(\psi,\alpha)]=0,\\
&\mathbb{E}[\widetilde{\Upsilon}(\psi,\alpha)\widetilde{\Upsilon}(\phi,\alpha)]=\frac{4}{\nu+1}\cos(\psi-\phi),\\
&\mathbb{E}[\widetilde{\Upsilon}(\psi,\alpha)^4]=\frac{48}{(\nu+1)(\nu+3)},\\
&\mathbb{E}[|{\Upsilon}(\psi,\alpha)-\widetilde{\Upsilon}(\psi,\alpha)|^2]\leq\frac{16}{(\nu+1)(\nu+3)},\\&\mathbb{E}[|{\Upsilon}(\psi,\alpha)|^2]\leq\frac{8}{\nu+1} .
\end{align*}\end{lem}By rotational
invariance, we also have $\mathbb{E}[\alpha^3]=\mathbb{E}[|\alpha|^2\alpha]=0$, which implies \begin{align}\label{EX13} \mathbb{E}[(\widetilde{\Upsilon}(\psi,\alpha)-\widetilde{\Upsilon}(\phi,\alpha))^3]=0.\end{align} We apply Plancharel's theorem to the power series of $\Upsilon$ to get \begin{align}\label{EX1X2}
&\mathbb{E}[({\Upsilon}(\psi,\alpha)-\widetilde{\Upsilon}(\psi,\alpha))\widetilde{\Upsilon}(\phi,\alpha)]=0.
\end{align}We also have the following estimate on $\Upsilon_1 $ (see Proposition 2.3 in \cite{KS09}), \begin{align}\label{r1}
&\left|\int_0^{2\pi}{\Upsilon}_1(\psi,re^{i\theta}){\Upsilon}_1(\phi,re^{i\theta})\frac{\mathrm{d}\theta}{2\pi}\right|\leq 2(|\psi|+|\phi|)\ln\frac{1}{1-r^2},\ \forall\ r\in[0,1).
\end{align}The following estimates are  proved in  Corollary 2.4 in \cite{KS09}. \begin{lem}\label{lem14}For $s\geq 0,$ we have\begin{align*}
&\mathbb{E}[\psi_s(\theta)]=(s+1)\theta,\,\,\,\,\mathbb{E}[|\psi_s(\theta)|]=(s+1)|\theta|,
\end{align*}and for $0\leq k\leq m,$ we have\begin{align*}
&\mathbb{E}[|\psi_m(\theta)-\psi_k(\theta)-(m-k)\theta|^2]=\sum_{s=k}^{m-1}\mathbb{E}[|{\Upsilon}_1(\psi_s(\theta),\gamma_s)|^2].
\end{align*}\end{lem}

By Lemma \ref{lem14}, we can further prove
\begin{lem}\label{lem15}For $0\leq k\leq m,$ we have\begin{align*}
&\mathbb{E}[|\psi_m(\theta)-\psi_k(\theta)-(m-k)\theta|^2]\leq 8(m-k)|\theta|/\beta.
\end{align*}In particular, for $k=0$ we have $$\mathbb{E}[|\psi_m(\theta)-(m+1)\theta|^2]\leq 8m|\theta|/\beta.$$ \end{lem}\begin{proof}Using Lemma \ref{lem14},  by \eqref{EXm}, \eqref{r1} and the rotational
invariance we have\begin{align*}
&\mathbb{E}[|\psi_m(\theta)-\psi_k(\theta)-(m-k)\theta|^2]=\sum_{s=k}^{m-1}\mathbb{E}[|{\Upsilon}_1(\psi_s(\theta),\gamma_s)|^2]\\ \leq&\sum_{s=k}^{m-1}
4\mathbb{E}\left[|\psi_s(\theta)|\ln\frac{1}{1-|\gamma_s|^2}\right]=\sum_{s=k}^{m-1}
4\mathbb{E}\left[|\psi_s(\theta)|\right]\mathbb{E}\left[\ln\frac{1}{1-|\gamma_s|^2}\right]\\ \leq&\sum_{s=k}^{m-1}
4(s+1)|\theta|\frac{2}{\beta(s+1)}=\sum_{s=k}^{m-1}
\frac{8|\theta|}{\beta}=\frac{8(m-k)|\theta|}{\beta}.
\end{align*}Here, we take $\nu=\beta(s+1)+1$ for $\gamma_s. $ This completes the proof.\end{proof}

\section{The characteristic function and the uniform variance bound}\label{the}
Let's define the characteristic function  of the Pr\"ufer phases $$a_k(\lambda):=a_k(\lambda,\theta,\beta)=\mathbb{E}[e^{i\lambda (\psi_{k}(\theta)-(k+1)\theta)}].$$  Then $|a_k(\lambda)|\leq 1$ for $ \lambda\in\R$.  In this section, we will derive several estimates regarding the sequence $\{a_k(\lambda)\}_{k=0}^{+\infty}$, then we can prove Lemma
\ref{lem4} and Corollary \ref{cor1}.

We first have
\begin{lem}\label{lem3}Suppose $ \phi,\lambda\in\R$ and $\alpha\sim\Theta_{\nu}$, then \begin{align*}
&\left|\mathbb{E}[e^{i\lambda {\Upsilon}_1(\phi,\alpha)}]-1+\frac{4\lambda^2(1-\cos\phi)}{\nu+1}\right|\leq
\frac{64\lambda^2+416\lambda^4}{(\nu+1)(\nu+3)}
\end{align*}and $$\left|\mathbb{E}[e^{i\lambda {\Upsilon}_1(\phi,\alpha)}]-1\right|\leq
\frac{16\lambda^2}{\nu+1} .$$\end{lem}\begin{proof}  Let $X={\Upsilon}_1(\phi,\alpha),\  X_1={\widetilde{\Upsilon}}(\phi,\alpha)-{\widetilde{\Upsilon}}(0,\alpha),\ X_3={{\Upsilon}}(\phi,\alpha)-{\widetilde{\Upsilon}}(\phi,\alpha),\ X_4={\widetilde{\Upsilon}}(0,\alpha)-{{\Upsilon}}(0,\alpha),\ X_2=X_3+X_4$, here $ \widetilde{\Upsilon}(\psi,\alpha)=2{\rm Im}[\alpha e^{i\psi}]$ for every $ \psi\in\R,$ then we have $X=X_1+X_3+X_4=X_1+X_2$. By Lemma \ref{lem2} we have\begin{align}\label{X1}
&\mathbb{E}[X_1]=\mathbb{E}[X_3]=\mathbb{E}[X_4]=0,\quad\mathbb{E}[X_2]=\mathbb{E}[X]=0,\\ \label{X3}
&\mathbb{E}[|X_3|^2]\leq\frac{16}{(\nu+1)(\nu+3)},\quad\mathbb{E}[|X_4|^2]\leq\frac{16}{(\nu+1)(\nu+3)},\\ \label{X2}
&\mathbb{E}[|X_2|^2]\leq 2\mathbb{E}[|X_3|^2+|X_4|^2]\leq\frac{64}{(\nu+1)(\nu+3)},\\ \label{X14}
&\mathbb{E}[|X_1|^4]\leq 8\mathbb{E}[\widetilde{\Upsilon}(0,\alpha)^4+\widetilde{\Upsilon}(\phi,\alpha)^4]\leq\frac{768}{(\nu+1)(\nu+3)},\\ \label{X12}
&\mathbb{E}[|X_1|^2]=\frac{8(1-\cos\phi)}{\nu+1},\quad\mathbb{E}[|X|^2]\leq\frac{32}{\nu+1}.\end{align} By \eqref{EX13}, \eqref{EX1X2} we have $ \mathbb{E}[X_1^3]=\mathbb{E}[X_1X_2]=0.$ Notice that\begin{align*}
&e^{i\lambda X}-e^{i\lambda X_1}=e^{i\lambda X_1}(e^{i\lambda X_2}-1)\\=&e^{i\lambda X_1}(e^{i\lambda X_2}-i\lambda X_2-1)+(e^{i\lambda X_1}-i\lambda X_1-1)(i\lambda X_2)+i\lambda X_2(i\lambda X_1+1),\end{align*} and that $|e^{ix}-ix-1|\leq|x|^2/2 $ for $x\in\R$ by Taylor expansion, we have\begin{align*}&|e^{i\lambda X}-e^{i\lambda X_1}-i\lambda X_2(i\lambda X_1+1)|\leq |\lambda X_2|^2/2+|\lambda X_1|^2|\lambda X_2|/2,\end{align*} which together with $\mathbb{E}[X_2]=\mathbb{E}[X_1X_2]=0 $ and \eqref{X2}, \eqref{X14} gives\begin{align}\label{EX-EX1}&|\mathbb{E}[e^{i\lambda X}-e^{i\lambda X_1}]|=|\mathbb{E}[e^{i\lambda X}-e^{i\lambda X_1}-i\lambda X_2(i\lambda X_1+1)]|\\ \nonumber\leq& \mathbb{E}[|\lambda X_2|^2/2+|\lambda X_1|^2|\lambda X_2|/2]\leq \mathbb{E}[|\lambda X_2|^2+|\lambda X_1|^4/2]\\ \nonumber \leq&\frac{64\lambda^2}{(\nu+1)(\nu+3)}+\frac{384\lambda^4}{(\nu+1)(\nu+3)}.\end{align} Since $|e^{ix}+ix^3/6+x^2/2-ix-1|\leq|x|^4/24 $ for $x\in\R$ by Taylor expansion and $ \mathbb{E}[X_1^3]=\mathbb{E}[X_1]=0$, by \eqref{X14} we have\begin{align*}
|\mathbb{E}[e^{i\lambda X_1}]-1+\lambda^2\mathbb{E}[|X_1|^2]/2|&=|\mathbb{E}[e^{i\lambda X_1}+i\lambda^3 X_1^3/6+\lambda^2 X_1^2/2-i\lambda X_1-1]|\\&\leq\mathbb{E}[|\lambda X_1|^4]/24\leq \frac{32\lambda^4}{(\nu+1)(\nu+3)},\end{align*} which together with  \eqref{X12}, \eqref{EX-EX1} gives\begin{align*}
&\left|\mathbb{E}[e^{i\lambda X}]-1+\frac{4\lambda^2(1-\cos\phi)}{\nu+1}\right|\\ \leq&|\mathbb{E}[e^{i\lambda X}-e^{i\lambda X_1}]|+|\mathbb{E}[e^{i\lambda X_1}]-1+\lambda^2\mathbb{E}[|X_1|^2]/2|\\ \leq&
\frac{64\lambda^2+384\lambda^4}{(\nu+1)(\nu+3)}+\frac{32\lambda^4}{(\nu+1)(\nu+3)}=\frac{64\lambda^2+416\lambda^4}{(\nu+1)(\nu+3)}.
\end{align*}This completes the proof of the first inequality. Since $|e^{ix}-ix-1|\leq|x|^2/2 $ for $x\in\R$ by Taylor expansion and $ \mathbb{E}[X]=0$, by \eqref{X12} we have\begin{align*}
|\mathbb{E}[e^{i\lambda X}]-1|&=|\mathbb{E}[e^{i\lambda X}-i\lambda X-1]|\leq \mathbb{E}[|\lambda X|^2/2]\leq\frac{16\lambda^2}{\nu+1}.\end{align*}This completes the proof of the second inequality.\end{proof}
We need the following estimate of the sequence $\{a_k(\lambda)\}_{k=0}^{+\infty}. $
 \begin{lem}\label{lem5} Let $ \theta\in(0,\pi],\ \lambda\in\R$, then we have \begin{align*}
&\left|a_{k+1}(\lambda)-a_k(\lambda)+\frac{\lambda^2(4a_k(\lambda)-2e^{i(k+1)\theta}a_k(\lambda+1)-2e^{-i(k+1)\theta}a_k(\lambda-1))}
{\beta(k+1)+2}\right|\\&\leq
\frac{64\lambda^2+416\lambda^4}{(\beta(k+1)+2)(\beta(k+1)+4)}
\end{align*}and $$\left|a_{k+1}(\lambda)-a_k(\lambda)\right|\leq
\frac{16\lambda^2}{\beta(k+1)+2} .$$\end{lem}\begin{proof}Let $X_k:=\psi_{k}(\theta),\ Y_k:=\Upsilon_1(\psi_{k}(\theta),\gamma_k)$, then $$X_{k+1}=X_k+\theta+Y_k.$$ For the sigma algebras $$\mathcal{M}_{k-1}:=\sigma(\gamma_0,\cdots,\gamma_{k-1}) ,$$ $\gamma_k $ is independent of $\mathcal{M}_{k-1} $ and $X_k$ is measurable in $\mathcal{M}_{k-1} .$ By Lemma \ref{lem3} we have (taking $\nu=\beta(k+1)+1$ for $\gamma_k$)\begin{align}\label{EYk}
&\left|\mathbb{E}[e^{i\lambda Y_k}|\mathcal{M}_{k-1}]-1+\frac{4\lambda^2(1-\cos X_k)}{\beta(k+1)+2}\right|\leq
\frac{64\lambda^2+416\lambda^4}{(\beta(k+1)+2)(\beta(k+1)+4)}
\end{align}and $$\left|\mathbb{E}[e^{i\lambda Y_k}|\mathcal{M}_{k-1}]-1\right|\leq
\frac{16\lambda^2}{\beta(k+1)+2} .$$ Let's denote $$Z_k(\lambda):=e^{i\lambda (X_{k}-(k+1)\theta)},$$ then  we have $$Z_k(\lambda)=e^{i\lambda (\psi_{k}(\theta)-(k+1)\theta)},\,\,Z_{k+1}(\lambda)=e^{i\lambda Y_k}Z_k(\lambda),\,\,a_k(\lambda)=\mathbb{E}[Z_k(\lambda)]$$ and\begin{align*}
2Z_k(\lambda)\cos X_k&=Z_k(\lambda)(e^{iX_{k}}+e^{-iX_{k}})=e^{i(\lambda+1) X_{k}-i\lambda(k+1)\theta}+e^{i(\lambda-1) X_{k}-i\lambda(k+1)\theta}\\&=Z_k(\lambda+1)e^{i(k+1)\theta}+Z_k(\lambda-1)e^{-i(k+1)\theta}.
\end{align*}Let's denote \begin{align*}
&V_k(\lambda):=Z_{k+1}(\lambda)-Z_k(\lambda)+Z_k(\lambda)\frac{4\lambda^2(1-\cos X_k)}{\beta(k+1)+2}\\=&e^{i\lambda Y_k}Z_k(\lambda)-Z_k(\lambda)+Z_k(\lambda)\frac{4\lambda^2(1-\cos X_k)}{\beta(k+1)+2}\\=&Z_{k+1}(\lambda)-Z_k(\lambda)+\frac{\lambda^2(4Z_k(\lambda)-2Z_k(\lambda+1)e^{i(k+1)\theta}-2Z_k(\lambda-1)e^{-i(k+1)\theta})}{\beta(k+1)+2}.
\end{align*}Then by \eqref{EYk} and the fact that $ |Z_k(\lambda)|=1$, $a_k(\lambda)=\mathbb{E}[Z_k(\lambda)] $ we have\begin{align*}
&\left|\mathbb{E}\left[V_k(\lambda)|\mathcal{M}_{k-1}\right]\right|=\left|\mathbb{E}[e^{i\lambda Y_k}|\mathcal{M}_{k-1}]Z_k(\lambda)-Z_k(\lambda)+Z_k(\lambda)\frac{4\lambda^2(1-\cos X_k)}{\beta(k+1)+2}\right|\\=&|Z_k(\lambda)|\left|\mathbb{E}[e^{i\lambda Y_k}|\mathcal{M}_{k-1}]-1+\frac{4\lambda^2(1-\cos X_k)}{\beta(k+1)+2}\right|\leq
\frac{64\lambda^2+416\lambda^4}{(\beta(k+1)+2)(\beta(k+1)+4)}\end{align*} and \begin{align*}&\left|a_{k+1}(\lambda)-a_k(\lambda)+\frac{\lambda^2(4a_k(\lambda)-2e^{i(k+1)\theta}a_k(\lambda+1)
-2e^{-i(k+1)\theta}a_k(\lambda-1))}{\beta(k+1)+2}\right|\\=&|\mathbb{E}[V_k(\lambda)]|\leq \mathbb{E}\left|\mathbb{E}\left[V_k(\lambda)|\mathcal{M}_{k-1}\right]\right|\leq
\frac{64\lambda^2+416\lambda^4}{(\beta(k+1)+2)(\beta(k+1)+4)},
\end{align*}which is the first inequality. Similarly, we have \begin{align*}
&\left|a_{k+1}(\lambda)-a_k(\lambda)\right|=|\mathbb{E}[Z_{k+1}(\lambda)-Z_k(\lambda)]|\leq \mathbb{E}\left|\mathbb{E}\left[Z_{k+1}(\lambda)-Z_k(\lambda)|\mathcal{M}_{k-1}\right]\right|\\
=&\mathbb{E}\left|\mathbb{E}\left[e^{i\lambda Y_k}|\mathcal{M}_{k-1}\right]Z_k(\lambda)-Z_k(\lambda)\right|=\mathbb{E}\left|\mathbb{E}[e^{i\lambda Y_k}|\mathcal{M}_{k-1}]-1\right|\leq
\frac{16\lambda^2}{\beta(k+1)+2},
\end{align*}which is the second inequality. This completes the proof.\end{proof}\begin{lem}\label{lem6} Let $ \theta\in(0,\pi],\ \delta\in[-\pi,\pi]\setminus\{0\},\ \lambda\in\R$, then we have\begin{align*}
&\left|\sum_{j=k}^{n-1}\frac{e^{ij\delta}a_j(\lambda)}{\beta(j+1)+2}\right|\leq\frac{2+
16\lambda^2/\beta}{|1-e^{i\delta}|(\beta (k+1)+2)}.
\end{align*}\end{lem}\begin{proof}Let $ \epsilon_j=1/(\beta (j+1)+2),\ a_j=a_j(\lambda)$, using summation by parts\begin{align*}
&\sum_{j=k}^{n-1}e^{ij\delta}\epsilon_ja_j(\lambda)=\sum_{j=k}^{n-1}\frac{e^{ij\delta}\epsilon_ja_j(\lambda)
-e^{i(j+1)\delta}\epsilon_ja_j(\lambda)}{1-e^{i\delta}}\\=&\frac{e^{ik\delta}\epsilon_ka_k-e^{in\delta}\epsilon_na_n}{1-e^{i\delta}}+
\sum_{j=k}^{n-1}\frac{(\epsilon_{j+1}-\epsilon_{j})e^{i(j+1)\delta}a_j}{1-e^{i\delta}}
+\sum_{j=k+1}^{n}\epsilon_j\frac{e^{ij\delta}(a_{j}-a_{j-1})}{1-e^{i\delta}},
\end{align*}and using $|a_j(\lambda)|\leq 1 $ we have\begin{align*}
&\left|\sum_{j=k}^{n-1}e^{ij\delta}\epsilon_ja_j\right|\leq\frac{|\epsilon_k|+|\epsilon_n|+
\sum_{j=k+1}^{n}|\epsilon_j-\epsilon_{j-1}|+\sum_{j=k+1}^{n}|\epsilon_j(a_j(\lambda)-a_{j-1}(\lambda))|}{|1-e^{i\delta}|}.
\end{align*}Since $\epsilon_{j-1}>\epsilon_{j}>0, $ we have $$\sum_{j=k+1}^{n}|\epsilon_j-\epsilon_{j-1}|=\epsilon_k-\epsilon_n $$ and\begin{align*}
&|\epsilon_k|+|\epsilon_n|+
\sum_{j=k+1}^{n}|\epsilon_j-\epsilon_{j-1}|=2\epsilon_k.
\end{align*}By Lemma \ref{lem5} we have $$|a_j(\lambda)-a_{j-1}(\lambda)|\leq
16\lambda^2 \epsilon_{j-1},$$ this together with $ \epsilon_{j-1}-\epsilon_{j}=\beta\epsilon_{j-1}\epsilon_{j}>0$ implies that\begin{align*}
&\sum_{j=k+1}^{n}|\epsilon_j(a_j(\lambda)-a_{j-1}(\lambda))|\leq\sum_{j=k+1}^{n}16\lambda^2|\epsilon_j\epsilon_{j-1}|
=(16\lambda^2/\beta)(\epsilon_k-\epsilon_n).
\end{align*}Summing up we conclude that\begin{align*}
&\left|\sum_{j=k}^{n-1}e^{ij\delta}\epsilon_ja_j(\lambda)\right|\leq\frac{2\epsilon_k+(16\lambda^2/\beta)\epsilon_k}{|1-e^{i\delta}|}=
\frac{2+16\lambda^2/\beta}{|1-e^{i\delta}|(\beta (k+1)+2)}.
\end{align*}This completes the proof. \end{proof}\begin{lem}\label{lem7}Given complex valued sequences $\epsilon_j,\ a_j,\ b_j,\ c_j$ and $n\in\Z,\ n>0,\ \lambda\in\R$ such that $|a_j|\leq 1,\ \epsilon_j>0,$ $a_{j+1}-a_j+\lambda^2(\epsilon_ja_j-b_j)=c_j$, let $ s_k=\sum_{j=0}^{k-1}\epsilon_j,\ t_k=\sum_{j=k}^{n-1}b_j$, then we have (for $k\in[0,n-1]\cap\Z$)\begin{align*}
&\left|e^{\lambda^2s_k}a_k-e^{\lambda^2s_n}a_n\right|\leq e^{\lambda^2s_k}\lambda^2|t_k|+\sum_{j=k}^{n-1}e^{\lambda^2s_{j+1}}(|c_j|+\lambda^4\epsilon_j^2/2+\lambda^4|\epsilon_jt_j|).
\end{align*}\end{lem}\begin{proof}By the definition of $t_k$ we have $ b_j=t_j-t_{j+1}$, inserting this into the equation of $c_j$ we have $a_{j+1}+\lambda^2t_{j+1}-a_j+\lambda^2(\epsilon_ja_j-t_j)=c_j$. Let $ \widetilde{a}_j=a_{j}+\lambda^2t_{j}$ then\begin{align*}
&\widetilde{a}_{j+1}-e^{-\lambda^2\epsilon_j}\widetilde{a}_j=c_j+
(1-\lambda^2\epsilon_j-e^{-\lambda^2\epsilon_j})a_j+\lambda^2
(1-e^{-\lambda^2\epsilon_j})t_j.
\end{align*}Since $|1-x-e^{-x}|\leq|x|^2/2,\ |1-e^{-x}|\leq|x| $ for $x\geq 0$ by Taylor expansion and $|a_j|\leq 1 $, we have\begin{align*}
&|\widetilde{a}_{j+1}-e^{-\lambda^2\epsilon_j}\widetilde{a}_j|\leq|c_j|+
\lambda^4\epsilon_j^2/2+\lambda^4
|\epsilon_jt_j|.\end{align*}By the definition of $s_k$ we have $ s_{j+1}=s_j+\epsilon_j$, thus\begin{align*}
|e^{\lambda^2s_k}\widetilde{a}_k-e^{\lambda^2s_n}\widetilde{a}_n|&\leq \sum_{j=k}^{n-1}|e^{\lambda^2s_{j+1}}\widetilde{a}_{j+1}-e^{\lambda^2s_j}\widetilde{a}_j|
=\sum_{j=k}^{n-1}e^{\lambda^2s_{j+1}}|\widetilde{a}_{j+1}-e^{-\lambda^2\epsilon_j}\widetilde{a}_j|\\ \numberthis\label{skak}&\leq \sum_{j=k}^{n-1}e^{\lambda^2s_{j+1}}(|c_j|+\lambda^4\epsilon_j^2/2+\lambda^4|\epsilon_jt_j|).
\end{align*}Notice that $t_n=0,\ e^{\lambda^2s_k}\widetilde{a}_k-e^{\lambda^2s_n}\widetilde{a}_n
=e^{\lambda^2s_k}{a}_k-e^{\lambda^2s_n}{a}_n+e^{\lambda^2s_k}\lambda^2t_{k},$ and\begin{align*}
|e^{\lambda^2s_k}{a}_k-e^{\lambda^2s_n}{a}_n|&\leq e^{\lambda^2s_k}\lambda^2|t_k|+|e^{\lambda^2s_k}\widetilde{a}_k-e^{\lambda^2s_n}\widetilde{a}_n|,
\end{align*}which together with \eqref{skak} concludes the proof.\end{proof}\begin{lem}\label{lem8} Let $ \theta\in(0,\pi],\ \lambda\in\R$, $\epsilon_k=4/(\beta (k+1)+2),\ s_k=\sum_{j=0}^{k-1}\epsilon_j$,  then (for $n,k\in\Z,\ n>k\geq0$)\begin{align*}
&\left|e^{\lambda^2s_k}a_k(\lambda)-e^{\lambda^2s_n}a_n(\lambda)\right|\leq \frac{C}{\theta}(\lambda^2+\lambda^4)e^{\lambda^2s_k}\epsilon_k+\frac{C}{\theta}(\lambda^2+\lambda^6)
\sum_{j=k}^{n-1}e^{\lambda^2s_{j+1}}\epsilon_j^2,
\end{align*}here $C>0$ is a constant depending only on $ \beta$.\end{lem}\begin{proof}Let $ a_k=a_k(\lambda),\ b_k=\dfrac{2e^{i(k+1)\theta}a_k(\lambda+1)+2e^{-i(k+1)\theta}a_k(\lambda-1)}
{\beta(k+1)+2}$, and $c_k=a_{k+1}-a_k+\lambda^2(\epsilon_ka_k-b_k),$ by Lemma \ref{lem5} we have $|c_k|\leq (4\lambda^2+26\lambda^4)\epsilon_k^2$. We can write $ t_k=\sum_{j=k}^{n-1}b_j=2(t_{k,1}+t_{k,2})$ such that\begin{align*}
&t_{k,1}=
\sum_{j=k}^{n-1}\dfrac{e^{i(j+1)\theta}a_j(\lambda+1)}
{\beta(j+1)+2},\,\,\, t_{k,2}=
\sum_{j=k}^{n-1}\dfrac{e^{-i(j+1)\theta}a_j(\lambda-1)}
{\beta(j+1)+2}.
\end{align*}By Lemma \ref{lem6} we have\begin{align*}
&|t_{k,1}|\leq\frac{2+
16(\lambda+1)^2/\beta}{|1-e^{i\theta}|(\beta (k+1)+2)},\,\,\,|t_{k,2}|\leq\frac{2+
16(\lambda-1)^2/\beta}{|1-e^{-i\theta}|(\beta (k+1)+2)},
\end{align*}thus\begin{align*}
&|t_{k}|\leq2(|t_{k,1}|+|t_{k,2}|)\leq4\cdot\frac{2+
16(\lambda^2+1)/\beta}{|1-e^{i\theta}|(\beta (k+1)+2)}\leq \frac{C(\lambda^2+1)\epsilon_k}{|1-e^{i\theta}|}.
\end{align*}Summing up we have\begin{align*}
(|c_j|+\lambda^4\epsilon_j^2/2+\lambda^4|\epsilon_jt_j|)\leq(4\lambda^2+26\lambda^4)\epsilon_j^2+\lambda^4\epsilon_j^2/2
+\frac{C\lambda^4\epsilon_j^2
(\lambda^2+1)}{|1-e^{i\theta}|}\\ \leq C(\lambda^2+\lambda^4)\epsilon_j^2+C\lambda^4\epsilon_j^2(\lambda^2+1)/\theta\leq C(\lambda^2+\lambda^6)\epsilon_j^2/\theta.
\end{align*}By Lemma \ref{lem7} we have\begin{align*}
\left|e^{\lambda^2s_k}a_k(\lambda)-e^{\lambda^2s_n}a_n(\lambda)\right|&\leq e^{\lambda^2s_k}\lambda^2|t_k|+\sum_{j=k}^{n-1}e^{\lambda^2s_{j+1}}(|c_j|+\lambda^4\epsilon_j^2/2+\lambda^4|\epsilon_jt_j|)\\&\leq e^{\lambda^2s_k}\frac{C\lambda^2(\lambda^2+1)\epsilon_k}{|1-e^{i\theta}|}
+C\sum_{j=k}^{n-1}e^{\lambda^2s_{j+1}}(\lambda^2+\lambda^6)\epsilon_j^2/\theta.
\end{align*}This completes the proof.\end{proof}\begin{lem}\label{lem9} Let $\epsilon_k=4/(\beta (k+1)+2),\ s_k=\sum_{j=0}^{k-1}\epsilon_j$,  then $|s_k-(4/\beta)\ln(k+1)|\leq C$, here $C>0$ is a constant depending only on $ \beta$.\end{lem}\begin{proof}By definition we have $s_0=0$ and $s_k-s_{k-1}=\epsilon_{k-1}$ for $k\geq 1.$ Let $\widetilde{s}_k=s_k-(4/\beta)\ln(k+1) $, then we have $\widetilde{s}_0=0$ and $\widetilde{s}_k-\widetilde{s}_{k-1}=\epsilon_{k-1}-(4/\beta)\ln(1+1/k)$ for $k\geq 1.$ Thus\begin{align*}
&|\widetilde{s}_k-\widetilde{s}_{k-1}|\leq |\epsilon_{k-1}-4/(\beta k)|+|4/(\beta k)-(4/\beta)\ln(1+1/k)|\\=&|4/(\beta k+2)-4/(\beta k)|+(4/\beta)|\ln(1+1/k)-1/k|\\ \leq&8/(\beta k)^2+(4/\beta)(1/k)^2/2=(8/\beta^2+2/\beta)/k^2,
\end{align*}and \begin{align*}
|s_k-(4/\beta)\ln(k+1)|&=|\widetilde{s}_k|\leq\sum_{j=1}^k|\widetilde{s}_j-\widetilde{s}_{j-1}|\leq \sum_{j=1}^k(8/\beta^2+2/\beta)/j^2\\&\leq(8/\beta^2+2/\beta)(\pi^2/6).
\end{align*}This completes the proof.\end{proof}\begin{lem}\label{lem10} Let $ \theta\in(0,\pi],\ \lambda\in\R,\ \lambda^2\leq\beta/8$,  $\epsilon_k=4/(\beta (k+1)+2),\ s_k=\sum_{j=0}^{k-1}\epsilon_j$,  then (for $n,k\in\Z,\ n\geq k\geq0$)\begin{align*}
&|e^{\lambda^2s_k}a_k(\lambda)-e^{\lambda^2s_n}a_n(\lambda)|\leq C\lambda^2e^{\lambda^2s_k}/(\theta(k+1)),
\end{align*}here $C>0$ is a constant depending only on $ \beta$.\end{lem}\begin{proof}If $n=k$ the result is clearly true, now we assume $n> k\geq0 .$ By Lemma \ref{lem9} we have\begin{align*}
\sum_{j=k}^{n-1}e^{\lambda^2s_{j+1}}\epsilon_j^2&\leq C\sum_{j=k}^{n-1}(j+2)^{(4/\beta)\lambda^2}(j+1)^{-2}\leq C\sum_{j=k}^{n-1}(j+2)^{(4/\beta)\lambda^2-2}\\&\leq C(k+1)^{(4/\beta)\lambda^2-1}\leq Ce^{\lambda^2s_{k}}(k+1)^{-1}.
\end{align*}Here we used the fact that $\lambda^2\leq\beta/8,\ (4/\beta)\lambda^2\leq 1/2<1$, which also implies that $\lambda^2+\lambda^4\leq C\lambda^2,\ \lambda^2+\lambda^6\leq C\lambda^2 $. By Lemma \ref{lem8} we have\begin{align*}
\left|e^{\lambda^2s_k}a_k(\lambda)-e^{\lambda^2s_n}a_n(\lambda)\right|\leq \frac{C}{\theta}(\lambda^2+\lambda^4)e^{\lambda^2s_k}\epsilon_k+\frac{C}{\theta}(\lambda^2+\lambda^6)
\sum_{j=k}^{n-1}e^{\lambda^2s_{j+1}}\epsilon_j^2\\ \leq \frac{C}{\theta}\lambda^2e^{\lambda^2s_k}(k+1)^{-1}+\frac{C}{\theta}\lambda^2
e^{\lambda^2s_{k}}(k+1)^{-1}\leq \frac{C\lambda^2e^{\lambda^2s_k}}{\theta(k+1)}.
\end{align*}This completes the proof. \end{proof}\begin{lem}\label{lem16} Let $ \theta\in(0,\pi],\ \lambda\in\R$, then\begin{align*}
&|a_k(\lambda)-1|\leq 4\lambda^2k|\theta|/\beta.
\end{align*}\end{lem}\begin{proof}Let $X_k=\psi_{k}(\theta)-(k+1)\theta$. By Lemma \ref{lem14} and Lemma \ref{lem15} we have $\mathbb{E}[X_k]=0,\ \mathbb{E}[X_k^2]\leq 8k|\theta|/\beta$, which together with Taylor expansion $|e^{ix}-ix-1|\leq|x|^2/2 $ for $x\in\R$ gives\begin{align*}
&|a_k(\lambda)-1|=|\mathbb{E}[e^{i\lambda X_{k}}]-1|=|\mathbb{E}[e^{i\lambda X_{k}}-i\lambda X_{k}-1]|\leq \mathbb{E}[|\lambda X_{k}|^2/2]\leq4\lambda^2k|\theta|/\beta.
\end{align*} This completes the proof. \end{proof}
\subsection{Proof of Lemma \ref{lem4}} Now we are ready to prove Lemma \ref{lem4}.
The proof relies on Lemma \ref{lem10} and Lemma \ref{lem16} with $ n$ replaced by $n-1.$\begin{proof}Let's denote $$b_k(\lambda)=e^{-(4\lambda^2/{\beta})\ln(2+k\theta)},\,\,\epsilon_k=4/(\beta (k+1)+2),\ s_k=\sum_{j=0}^{k-1}\epsilon_j$$ for every $k\in\Z,\ k\geq 0.$

If $n\theta\leq 2, $ by Lemma \ref{lem16} we have $$|a_{n-1}(\lambda)-1|\leq 4\lambda^2(n-1)|\theta|/\beta\leq 8\lambda^2/\beta. $$ By Taylor expansion we have $$|b_{n}(\lambda)-1|\leq(4\lambda^2/{\beta})\ln(2+n\theta)\leq (4\lambda^2/{\beta})\ln4\leq (1/2)\ln4$$ and $$e^{(4\lambda^2/{\beta})\ln(2+n\theta)}\leq e^{(1/2)\ln4}=2.$$ Thus we have \begin{align*}
&|\mathbb{E}[e^{i\lambda (\psi_{n-1}(\theta)-n\theta)}]-e^{-(4\lambda^2/{\beta})\ln(2+n\theta)}|=|a_{n-1}(\lambda)-b_{n}(\lambda)|\\ \leq&|a_{n-1}(\lambda)-1|+|b_{n}(\lambda)-1|\leq 8\lambda^2/\beta+(4\lambda^2/{\beta})\ln4
 \\ \leq&(8+4\ln 4)(\lambda^2/\beta)(2e^{-(4\lambda^2/{\beta})\ln(2+n\theta)})\leq C\lambda^2e^{-(4\lambda^2/{\beta})\ln(2+n\theta)}.
\end{align*}If $n\theta\geq 2, $ we take $ k=\lfloor1/\theta\rfloor$, then we have $0\leq k\leq 1/\theta\leq n/2<n$, thus $k\leq n-1.$ By Lemma \ref{lem16} we have $$|a_{k}(\lambda)-1|\leq 4\lambda^2k|\theta|/\beta\leq 4\lambda^2/\beta. $$ By Lemma \ref{lem9} we have $$|(s_{n-1}-s_k)-(4/\beta)\ln(n/(k+1))|\leq C.$$ We also have $k\theta\leq1\leq (k+1)\theta\leq 1+\theta\leq 1+\pi,\ 0\leq \ln((k+1)\theta)\leq C,\ 0\leq \ln(2/(n\theta)+1)\leq \ln 2,$ and $|\ln(2+n\theta)-\ln(n/(k+1))|=|\ln(2/(n\theta)+1)+\ln((k+1)\theta)|\leq C,$ thus $$|(s_{n-1}-s_k)-(4/\beta)\ln(2+n\theta)|\leq C,$$  therefore, we have
$$|e^{\lambda^2(s_k-s_{n-1})}-e^{-(4\lambda^2/{\beta})\ln(2+n\theta)}|\leq C\lambda^2e^{-(4\lambda^2/{\beta})\ln(2+n\theta)}.$$
By Lemma \ref{lem10},
 we have $$|e^{\lambda^2s_k}a_k(\lambda)-e^{\lambda^2s_{n-1}}a_{n-1}(\lambda)|\leq C\lambda^2e^{\lambda^2s_k}/(\theta(k+1)),$$ and thus we have  \begin{align*}
&|e^{\lambda^2(s_k-s_{n-1})}a_k(\lambda)-a_{n-1}(\lambda)|\leq C\lambda^2e^{\lambda^2(s_k-s_{n-1})}/(\theta(k+1))\\ \leq& C\lambda^2e^{\lambda^2(s_k-s_{n-1})}\leq C\lambda^2e^{-(4\lambda^2/{\beta})\ln(2+n\theta)}. \end{align*}
Now we have (recall $b_n(\lambda)=e^{-(4\lambda^2/{\beta})\ln(2+n\theta)}$ and $|a_{k}(\lambda)-1|\leq 4\lambda^2/\beta$)\begin{align*}
|e^{\lambda^2(s_k-s_{n-1})}a_k(\lambda)-b_{n}(\lambda)|&\leq |e^{\lambda^2(s_k-s_{n-1})}-b_{n}(\lambda)|+|e^{\lambda^2(s_k-s_{n-1})}(a_k(\lambda)-1)|\\ &\leq C\lambda^2b_{n}(\lambda)+Cb_{n}(\lambda)|a_{k}(\lambda)-1|\leq C\lambda^2b_{n}(\lambda)\end{align*} and \begin{align*}&|e^{\lambda^2(s_k-s_{n-1})}a_k(\lambda)-a_{n-1}(\lambda)|\leq C\lambda^2b_{n}(\lambda).\end{align*}Therefore, we have  $$|a_{n-1}(\lambda)-b_{n}(\lambda)|\leq C\lambda^2b_{n}(\lambda).$$
Now the result follows by the definitions of $a_{n-1}(\lambda)$ and $b_{n}(\lambda) .$\end{proof}
\subsection{Proof of Corollary \ref{cor1}}
As a consequence of Lemma \ref{lem4}, we now give the proof of Corollary \ref{cor1}.
\begin{proof}Let $X=\psi_{n-1}(\theta)-n\theta,\ Z=\lfloor\frac{\psi_{n-1}(\theta)-\eta}{2\pi}\rfloor+1$, then $N_n(0,\theta)\eqd Z$. Taking the real part in Lemma \ref{lem4} we have\begin{align*}
&|\mathbb{E}[\cos({\lambda X})]-e^{-(4\lambda^2/{\beta})\ln(2+n\theta)}|\leq
 C\lambda^2e^{-(4\lambda^2/{\beta})\ln(2+n\theta)}\leq
 C\lambda^2,\\&|\mathbb{E}[(1-\cos(\lambda X))/\lambda^2]-(1-e^{-(4\lambda^2/{\beta})\ln(2+n\theta)})/\lambda^2|\leq
 C,\end{align*}for $ \lambda\in[-\sqrt{\beta/8},\sqrt{\beta/8}]\setminus\{0\}.$ Letting $\lambda\to 0$ we conclude that\begin{align*}
&|\mathbb{E}[X^2/2]-(4/{\beta})\ln(2+n\theta)|\leq
 C,\end{align*}which implies \eqref{ds}. Since $ \eta$ is a
uniform random variable on $[0,2\pi),$ we have \begin{align*}
&\mathbb{E}[\lfloor x-\eta/(2\pi)\rfloor+1]=\frac{1}{2\pi}\int_0^{2\pi}(\lfloor x-\eta/(2\pi)\rfloor+1) d\eta=\int_{x-1}^{x}(\lfloor y\rfloor+1) dy\\=&\int_{x-1}^{\lfloor x\rfloor}\lfloor x\rfloor dy+\int_{\lfloor x\rfloor}^{x}(\lfloor x\rfloor+1) dy=\lfloor x\rfloor(\lfloor x\rfloor-x+1)+(\lfloor x\rfloor+1)(x-\lfloor x\rfloor)=x,
\end{align*} for $x\in\R.$ Since $ \eta$ is independent of $\psi_{n-1}(\theta) $ and $N_n(0,\theta)\eqd Z=\lfloor\frac{\psi_{n-1}(\theta)-\eta}{2\pi}\rfloor+1 $, we have $\mathbb{E}[Z|\psi_{n-1}(\theta)]=\psi_{n-1}(\theta)/(2\pi).$ Let $Z_1:=Z-\psi_{n-1}(\theta)/(2\pi)$ then $|Z_1|\leq 1,\ \mathbb{E}[Z_1|\psi_{n-1}(\theta)]=0.$ For $ \lambda\in\R$ we first have\begin{align*}&|\mathbb{E}[e^{i\lambda ( Z-n\theta/(2\pi))}]-\mathbb{E}[e^{i\lambda ( \psi_{n-1}(\theta)-n\theta)/(2\pi)}]|=|\mathbb{E}[e^{i\lambda Z}-e^{i\lambda \psi_{n-1}(\theta)/(2\pi)}]|\\ \leq&\mathbb{E}|\mathbb{E}[e^{i\lambda Z}-e^{i\lambda \psi_{n-1}(\theta)/(2\pi)}|\psi_{n-1}(\theta)]|=\mathbb{E}|(\mathbb{E}[e^{i\lambda Z_1}|\psi_{n-1}(\theta)]-1)e^{i\lambda \psi_{n-1}(\theta)/(2\pi)}|\\=&|\mathbb{E}[e^{i\lambda Z_1}|\psi_{n-1}(\theta)]-1|
=|\mathbb{E}[e^{i\lambda Z_1}-1-i\lambda Z_1|\psi_{n-1}(\theta)]|\\ \leq& \mathbb{E}[(\lambda Z_1)^2/2|\psi_{n-1}(\theta)]\leq \lambda^2/2.
\end{align*}On the other hand, for $ \lambda\in[-2\pi\sqrt{\beta/8},2\pi\sqrt{\beta/8}],$ let $ \lambda_1=\lambda/(2\pi)$ then $ \lambda_1^2\leq \beta/8,$ by Lemma \ref{lem4} we have\begin{align*}&|\mathbb{E}[e^{i\lambda ( \psi_{n-1}(\theta)-n\theta)/(2\pi)}]-e^{-\lambda^2/({\beta}\pi^2)\cdot\ln(2+n\theta)}|\\=&|\mathbb{E}[e^{i\lambda_1 ( \psi_{n-1}(\theta)-n\theta)}]-e^{-(4\lambda_1^2/{\beta})\cdot\ln(2+n\theta)}|\leq C\lambda_1^2e^{-(4\lambda_1^2/{\beta})\cdot\ln(2+n\theta)}\leq C\lambda_1^2\leq C\lambda^2.
\end{align*}Therefore, we have \begin{align*}&|\mathbb{E}[e^{i\lambda ( Z-n\theta/(2\pi))}]-e^{-\lambda^2/({\beta}\pi^2)\cdot\ln(2+n\theta)}|\leq C\lambda^2,
\end{align*}which implies \eqref{uds}. We also have \begin{align}\label{EN}
&\mathbb{E}[ \left| N_n(0,\theta)-{n\theta}/({2\pi})\right|^2]=\mathbb{E}[|Z-n\theta/(2\pi)|^2]\\ \nonumber =&\mathbb{E}[(\psi_{n-1}(\theta)-n\theta)^2/(2\pi)^2]+\mathbb{E}[|Z-\psi_{n-1}(\theta)/(2\pi)|^2]\end{align} and \begin{align} \label{eq0}0\leq&\mathbb{E}[|Z-\psi_{n-1}(\theta)/(2\pi)|^2]=\mathbb E(|Z_1|^2)\leq 1.
\end{align}Using \eqref{ds}, \eqref{EN} and \eqref{eq0}, we conclude \eqref{udd}.\end{proof}

\section{Proof of Theorem \ref{prop2}}\label{theorem}
In this section, we will finish the proof of Theorem \ref{prop2}.

Let $F(x)$ be the distribution function of a random variable $X$ and let \begin{align}\label{Gx}G(x):=\frac{1}{\sqrt{2\pi}}\int_{-\infty}^xe^{-t^2/2}dt\end{align} be the Gaussian distribution function. Let's denote $$M=\sup_{x\in\R}|F(x)-G(x)|,\ \delta=M(\pi/2)^{1/2}$$ and let $$ \phi(t):=\mathbb{E}[e^{itX}]=\int_{\R}e^{itx}dF(x),\  \psi(t):=\int_{\R}e^{itx}dG(x)=e^{-t^2/2}$$ be the characteristic functions.

For every $T>0$ we have the following bound (see (30) in \cite{Ber})\begin{align}\label{30}&A(T\delta)\leq \int_0^T(T-t)\frac{|\phi(t)-\psi(t)|}{t}dt\leq T\int_0^T\frac{|\phi(t)-\psi(t)|}{t}dt,
\end{align}where\begin{align*}&A(u)=(2\pi)^{1/2}\cdot u\cdot\left(3\int_0^u\frac{1-\cos x}{x^2}dx-\pi\right).
\end{align*}Now we take $$T=\sqrt{\ln(2+n\theta)},\ X=\sqrt{{\beta}/8}(\psi_{n-1}(\theta)-n\theta)/T,$$ for $ \theta\in(0,\pi]$. Let ${a}_k(\lambda)=\mathbb{E}[e^{i\lambda (\psi_{k}(\theta)-(k+1)\theta)}]$ as in \S\ref{the},  then we have $$\phi(t)=\mathbb{E}[e^{itX}]={a}_{n-1}(t\sqrt{{\beta}/8}/T). $$ By Lemma \ref{lem4} we have (for $\lambda\in\R,\ \lambda^2\leq\beta/8 $)\begin{align*}
&|{a}_{n-1}(\lambda)-e^{-(4\lambda^2/{\beta})\ln(2+n\theta)}|\leq
 C\lambda^2e^{-(4\lambda^2/{\beta})\ln(2+n\theta)}.
\end{align*}Notice that if $t\in[0,T],\ \lambda=t\sqrt{{\beta}/8}/T,$ then $\lambda^2\leq\beta/8,\ (4\lambda^2/{\beta})\ln(2+n\theta)=t^2/2 .$ Thus we have \begin{align*}&|\phi(t)-\psi(t)|=|{a}_{n-1}(t\sqrt{{\beta}/8}/T)-e^{-t^2/2}|\leq C{\beta}t^2/(8T^2)\cdot e^{-t^2/2}
\end{align*}and \begin{align*}&T\int_0^T\frac{|\phi(t)-\psi(t)|}{t}dt\leq CT\int_0^T{\beta}t^2/(8T^2)\cdot e^{-t^2/2}dt\leq C/T\leq C.
\end{align*}By \eqref{30} we have $A(T\delta)\leq C$. As $\lim\limits_{u\to+\infty}A(u)=+\infty$, we have $T\delta\leq C. $ Recall that $ \delta=M(\pi/2)^{1/2},$ we have $ \delta\leq C/T,\ M\leq C/T.$ Recall that $T=\sqrt{\ln(2+n\theta)},\ M=\sup_{x\in\R}|F(x)-G(x)|,\ F(x)=\mathbb{P}[X\leq x],\ X=\sqrt{{\beta}/8}(\psi_{n-1}(\theta)-n\theta)/T,$ now we have proven the following result.\begin{lem}\label{lem17} Let $ \theta\in(0,\pi],\ n>0,\ n\in\Z$,  then \begin{align*}\sup_{x\in\R}|\mathbb{P}[\sqrt{{\beta}/(8\ln(2+n\theta))}(\psi_{n-1}(\theta)-n\theta)\leq x]-G(x)|\leq C(\ln(2+n\theta))^{-1/2}.\end{align*}Here, $C>0$ is  a constant depending only on $ \beta$.\end{lem}
\subsection{Proof of Theorem \ref{prop2}}Now we give the proof of Theorem \ref{prop2}.\begin{proof}Since $N_n(0,\theta)\eqd Z$ for $Z=\lfloor\frac{\psi_{n-1}(\theta)-\eta}{2\pi}\rfloor+1$ and $ \eta\in[0,2\pi)$, we have $|Z-\frac{\psi_{n-1}(\theta)}{2\pi}|\leq 1.$ Let  $T_1=\sqrt{\frac{\pi^2\beta}{2\ln(2+n\theta)}}$,  then we have $T_1/(2\pi)=\sqrt{{\beta}/(8\ln(2+n\theta))}.$ Thus for $x\in\R$, by Lemma \ref{lem17} we have\begin{align*}
&\mathbb{P}[T_1(N_n(0,\theta)-{n\theta}/{(2\pi)})\leq x]=\mathbb{P}[T_1(Z-{n\theta}/{(2\pi)})\leq x]\\ \leq&\mathbb{P}[T_1(\psi_{n-1}(\theta)/(2\pi)-1-{n\theta}/{(2\pi)})\leq x]\\=&\mathbb{P}[T_1/(2\pi)\cdot(\psi_{n-1}(\theta)-{n\theta})\leq x+T_1]\leq G(x+T_1)+ C(\ln(2+n\theta))^{-1/2}\\ \leq &G(x)+T_1/\sqrt{2\pi}+ C(\ln(2+n\theta))^{-1/2}\leq G(x)+ C(\ln(2+n\theta))^{-1/2},
\end{align*}here we used the fact that $0\leq G'(x)=e^{-x^2/2}/\sqrt{2\pi}$ for $x\in\R$ which implies that $|G(x)-G(y)|\leq |x-y|/\sqrt{2\pi}$ for $x,y\in\R.$ Similarly, we have \begin{align*}
&\mathbb{P}[T_1(N_n(0,\theta)-{n\theta}/{(2\pi)})\leq x]\geq\mathbb{P}[T_1/(2\pi)\cdot(\psi_{n-1}(\theta)-{n\theta})\leq x-T_1]\\ \geq& G(x-T_1)- C(\ln(2+n\theta))^{-1/2}\geq G(x)- C(\ln(2+n\theta))^{-1/2}.
\end{align*} Combining the upper and lower bounds we conclude that\begin{align*}
&\sup_{x\in\R}|\mathbb{P}[T_1(N_n(0,\theta)-{n\theta}/{(2\pi)})\leq x]- G(x)|\leq C(\ln(2+n\theta))^{-1/2}.
\end{align*}This completes the proof of Theorem \ref{prop2} by the definitions of $T_1$ and $G(x)$.\end{proof}

\section{Results for Sine${_\beta}$ process}\label{cool}
Now we give the proof of Corollary \ref{cor3}.

\begin{proof}Since the Sine$_\beta$ point process is the scaling limit of the C$\beta$E, by Skorokhod's representation theorem, one can construct point processes $L_n$ and $L$ such that  the point measure
corresponding to $L_n$ converges locally weakly to the measure corresponding to $ L$ almost surely \cite{NV}, and $$\text{Card}(L_n\cap[0,x])\eqd N_n(0,x/n) , \,\,\,0<x<2\pi n.$$ Let $x>0,\ \lambda\in[-2\pi\sqrt{\beta/8},2\pi\sqrt{\beta/8}].$ Since $L$ almost surely does not contain the points $ 0$ and $x$, we have almost surely\begin{align*}
&\text{Card}(L_n\cap[0,x])\to\text{Card}(L\cap[0,x]),\end{align*} and \begin{align*}e^{i\lambda ( \text{Card}(L_n\cap[0,x])-x/(2\pi))}\to e^{i\lambda ( \text{Card}(L\cap[0,x])-x/(2\pi))}.
\end{align*} By dominated convergence theorem we have\begin{align*}
&\mathbb{E}[e^{i\lambda ( \text{Card}(L_n\cap[0,x])-x/(2\pi))}]\to \mathbb{E}[e^{i\lambda (\text{Card}(L\cap[0,x])-x/(2\pi))}].
\end{align*}For $n>x/\pi$ we have $\pi n>x$, and by \eqref{uds} in Corollary \ref{cor1} we have\begin{align*}&|\mathbb{E}[e^{i\lambda ( \text{Card}(L_n\cap[0,x])-x/(2\pi))}]-e^{-\lambda^2/({\beta}\pi^2)\cdot\ln(2+x)}|\\=&|\mathbb{E}[e^{i\lambda ( N_n(0,x/n)-x/(2\pi))}]-e^{-\lambda^2/({\beta}\pi^2)\cdot\ln(2+x)}|\leq C\lambda^2,
\end{align*}which implies\begin{align*}&|\mathbb{E}[e^{i\lambda ( \text{Card}(L\cap[0,x])-x/(2\pi))}]-e^{-\lambda^2/({\beta}\pi^2)\cdot\ln(2+x)}|\leq C\lambda^2.
\end{align*} Taking the real part we have\begin{align*}
&|\mathbb{E}[(1-\cos(\lambda ( \text{Card}(L\cap[0,x])-x/(2\pi))))/\lambda^2]-(1-e^{-\lambda^2/({\beta}\pi^2)\cdot\ln(2+x)})/\lambda^2|\leq
 C,\end{align*}for $ \lambda\in[-2\pi\sqrt{\beta/8},2\pi\sqrt{\beta/8}]\setminus\{0\}.$ Letting $\lambda\to 0$ we conclude that\begin{align*}
&|\mathbb{E}[( \text{Card}(L\cap[0,x])-x/(2\pi))^2/2]-1/({\beta}\pi^2)\cdot\ln(2+x)|\leq
 C,\end{align*}which implies \eqref{L1}.

 Now let $x>0,\ y\in\R$, $X_n=\text{Card}(L_n\cap[0,x]),\ X=\text{Card}(L\cap[0,x]),\ T_1=\sqrt{\frac{\pi^2\beta}{2\ln(2+x)}}$, then we have $X_n\to X$ almost surely.  For $n>x/\pi$ we have $\pi n>x$, and $X_n\eqd N_n(0,x/n), $ by Theorem \ref{prop2}, we have\begin{align*}
&|\mathbb{P}[T_1(X_n-x/{(2\pi)})\leq y]- G(y)|\leq C(\ln(2+x))^{-1/2},
\end{align*}where the function $G$ is defined in \eqref{Gx}. For every $a>0$ we have\begin{align*}
\mathbb{P}[T_1(X-x/{(2\pi)})\leq y]&\leq\liminf_{n\to+\infty}\mathbb{P}[T_1(X_n-x/{(2\pi)})\leq y+a]\\&\leq G(y+a)+ C(\ln(2+x))^{-1/2}.
\end{align*}Since $G$ is continuous we have\begin{align*}
\mathbb{P}[T_1(X-x/{(2\pi)})\leq y]&\leq G(y)+ C(\ln(2+x))^{-1/2}.\end{align*}Similarly, we have \begin{align*}\mathbb{P}[T_1(X-x/{(2\pi)})\leq y]&\geq G(y)- C(\ln(2+x))^{-1/2}.
\end{align*} Combining the upper and lower bounds we conclude that\begin{align*}
&\sup_{y\in\R}|\mathbb{P}[T_1(X-x/{(2\pi)})\leq y]- G(y)|\leq C(\ln(2+x))^{-1/2},
\end{align*}which gives \eqref{L2} by the definitions of $T_1,\ X$ and $G(y)$. This completes the proof.\end{proof}
\section{Application: normality of linear statistics}\label{linear}

In this section, we will prove Theorem \ref{lem19}.
\subsection{Variance bound}
We first need the following estimate on the variance of the linear statistics. We write $\|g\|_{L^p}=\|g\|_{L^p(0,2\pi)}$.
 \begin{lem}\label{lem18} Let $f\in W^{1,p}(S^1)$ be real valued and $\int_0^{2\pi}f(x)dx=0$, then\begin{align*}
&\mathbb{E}[\langle\xi_n,f\rangle]=0,\,\, \mathbb{E}|\langle\xi_n,f\rangle|^2\leq C\|f'\|_{L^p}^2,
\end{align*}here $p\in(1,+\infty),$ and $C>0$ is a constant depending only on $\beta,p.$\end{lem}

To prove Lemma \ref{lem18}, we first need the following lemma which is the consequence of the uniform variance bound \eqref{udd} in Corollary \ref{cor1}.

Let $\widetilde{N}_n(a,b)={N}_n(a,b)-n(b-a)/(2\pi),\ \widetilde{N}_n(b,a)=-\widetilde{N}_n(a,b)$ for $0\leq a\leq b< 2\pi.$ As $N_n(a,b)=N_n(0,b)-N_n(0,a),$ for $0\leq a\leq b< 2\pi,$ we have $\widetilde{N}_n(a,b)=\widetilde{N}_n(0,b)-\widetilde{N}_n(0,a)$ for $a,b\in[0,2\pi).$\begin{lem}\label{lem20}For $a,b\in[0,2\pi),\ a\neq b$ we have\begin{align*}
&|\mathbb{E}[\widetilde{N}_n(a,b)^2]-2\ln n/(\pi^2\beta)|\leq C(1-\ln\sin(|a-b|/2)),
\end{align*}here $C>0$ is a constant depending only on $\beta.$\end{lem}\begin{proof}By symmetry we only need to consider the case $0\leq a<b< 2\pi.$ For $x\in(0,\pi],$ by \eqref{udd} in Corollary \ref{cor1} we have\begin{align*}|\mathbb{E}[\widetilde{N}_n(0,x)^2]-{2\ln(2+nx)}/({\pi^2\beta})|\leq C.
\end{align*}Thus we have\begin{align*}
\mathbb{E}[\widetilde{N}_n(0,x)^2]&\leq{2\ln(2+nx)}/({\pi^2\beta})+C \\&\leq{2\ln(6n)}/({\pi^2\beta})+C\leq{2\ln n}/({\pi^2\beta})+C
\end{align*}and\begin{align*}
\mathbb{E}[\widetilde{N}_n(0,x)^2]&\geq{2\ln(2+nx)}/({\pi^2\beta})-C\geq{2\ln(nx)}/({\pi^2\beta})-C\\&\geq{2\ln n}/({\pi^2\beta})+{2\ln (2\sin(x/2))}/({\pi^2\beta})-C\\&\geq{2\ln n}/({\pi^2\beta})-C(1-\ln\sin(x/2)),
\end{align*}here we used the fact that $\sin(x/2)\leq1,\ 2 \sin(x/2)\leq x,\ \ln \sin(x/2)\leq 0.$ Combining the upper and lower bounds we conclude that\begin{align}\label{N0x}
&|\mathbb{E}[\widetilde{N}_n(0,x)^2]-2\ln n/(\pi^2\beta)|\leq C(1-\ln\sin(x/2))
\end{align}for $x\in(0,\pi].$ If $x\in[\pi,2\pi),$ by rotational invariance we have $\widetilde{N}_n(0,x)=-\widetilde{N}_n(x,2\pi)\\ \eqd-\widetilde{N}_n(0,2\pi-x),\ 2\pi-x\in(0,\pi]$ and \begin{align*}
&|\mathbb{E}[\widetilde{N}_n(0,x)^2]-2\ln n/(\pi^2\beta)|=|\mathbb{E}[\widetilde{N}_n(0,2\pi-x)^2]-2\ln n/(\pi^2\beta)|\\ \leq&C(1-\ln\sin((2\pi-x)/2))=C(1-\ln\sin(x/2)).
\end{align*} Thus \eqref{N0x} is true for $x\in (0,2\pi).$ Now for $0\leq a<b<2\pi,$ by rotational invariance we have $\widetilde{N}_n(a,b)\eqd\widetilde{N}_n(0,b-a),$ and by \eqref{N0x} we have\begin{align*}
|\mathbb{E}[\widetilde{N}_n(a,b)^2]-2\ln n/(\pi^2\beta)|=&|\mathbb{E}[\widetilde{N}_n(0,b-a)^2]-2\ln n/(\pi^2\beta)|\\ \leq&C(1-\ln\sin((b-a)/2)).
\end{align*} This completes the proof.\end{proof}Now we give the proof of Lemma \ref{lem18}.\begin{proof}By definition and $\int_0^{2\pi}f(x)dx=0$ and integration by parts we have \begin{align*}
\langle\xi_n,f\rangle&=\int_0^{2\pi}f(x)dN_n(0,x)=\int_0^{2\pi}f(x)d(N_n(0,x)-nx/(2\pi))\\&=-\int_0^{2\pi}f'(x)(N_n(0,x)-nx/(2\pi))dx.
\end{align*}By rotational invariance we have $\mathbb{E}[N_n(0,x)]=nx/(2\pi), $ which implies $\mathbb{E}[\langle\xi_n,f\rangle]\\=0. $ By the definition of $\widetilde{N}_n(a,b) $ we have\begin{align*}
&|\langle\xi_n,f\rangle|^2=\int_0^{2\pi}\int_0^{2\pi}f'(x)f'(y)\widetilde{N}_n(0,x)\widetilde{N}_n(0,y)dxdy\\
=&-\frac{1}{2}\int_0^{2\pi}\int_0^{2\pi}f'(x)f'(y)(\widetilde{N}_n(0,x)-\widetilde{N}_n(0,y))^2dxdy\\
&+\int_0^{2\pi}\int_0^{2\pi}f'(x)f'(y)\widetilde{N}_n(0,x)^2dxdy\\
=&-\frac{1}{2}\int_0^{2\pi}\int_0^{2\pi}f'(x)f'(y)\widetilde{N}_n(y,x)^2dxdy,
\end{align*}here we used the fact that $\int_0^{2\pi}f'(y)dy=f(2\pi)-f(0)=0, $ which also implies that \begin{align*}
&\mathbb{E}|\langle\xi_n,f\rangle|^2=-\frac{1}{2}\int_0^{2\pi}\int_0^{2\pi}f'(x)f'(y)\mathbb{E}[\widetilde{N}_n(y,x)^2]dxdy\\
=&-\frac{1}{2}\int_0^{2\pi}\int_0^{2\pi}f'(x)f'(y)(\mathbb{E}[\widetilde{N}_n(y,x)^2]-2\ln n/(\pi^2\beta))dxdy.
\end{align*}By Lemma \ref{lem20} we have \begin{align*}
\mathbb{E}|\langle\xi_n,f\rangle|^2&\leq\frac{1}{2}\int_0^{2\pi}\int_0^{2\pi}|f'(x)f'(y)||\mathbb{E}[\widetilde{N}_n(y,x)^2]-2\ln n/(\pi^2\beta)|dxdy\\&\leq C\int_0^{2\pi}\int_0^{2\pi}|f'(x)f'(y)|(1-\ln\sin(|x-y|/2))dxdy.
\end{align*}Notice that \begin{align*}
\int_0^{2\pi}(1-\ln\sin(|x-y|/2))^{p'}dy=\int_0^{2\pi}(1-\ln\sin(y/2))^{p'}dy=C_p<+\infty,
\end{align*}for $x\in[0,2\pi]$, where $p'=p/(p-1)$ and $C_p$ is a constant depending only on $p.$ By H\"older's inequality we have \begin{align*}
\int_0^{2\pi}|f'(y)|(1-\ln\sin(|x-y|/2))dy\leq\|f'\|_{L^p}C_p^{1-1/p}
\end{align*}for $x\in[0,2\pi]$, and\begin{align*}
\mathbb{E}|\langle\xi_n,f\rangle|^2&\leq C\|f'\|_{L^p}C_p^{1-1/p}\int_0^{2\pi}|f'(x)|dx\leq C\|f'\|_{L^p}C_p^{1-1/p}\|f'\|_{L^p}.
\end{align*}This completes the proof.\end{proof}

\subsection{Proof of Theorem \ref{lem19}}
Now we are ready to prove  Theorem \ref{lem19}.
The proof is based on the following result of Jiang-Matsumoto for the case $f(x)$ a finite sum of $ \{e^{ikx}\}_{k\in\Z}$ (see Corollary 3 in \cite{JM}).\begin{lem}\label{lem21} Let $(\theta_1,\cdots,\theta_n)$ be a sample from $ \mu_{\beta,n}$. Let $g(z)=\sum_{k=0}^mc_kz^k$ with fixed $m$ and $c_k\in\mathbb{C}$ for all $k$. Set $X_n=\sum_{j=1}^ng(e^{i\theta_j})$. then $X_n-\mu_n$  converges in law to a complex Gaussian random variable $\sim\mathbb{C}N(0, \sigma^2),$ where\begin{align*}
&\mu_n=nc_0,\,\, \sigma^2=\frac{2}{\beta}\sum_{j=1}^{+\infty}j|c_j|^2.
\end{align*}\end{lem}Lemma \ref{lem21} tells us that if $c_0=0,\ f(x)=g(e^{ix})+\overline{g(e^{ix})}$ then $X_n+\overline{X_n}=\langle\xi_n,f\rangle $ converges in law to a real Gaussian random variable $\sim N(0, 2\sigma^2).$

Now we give the proof of Theorem \ref{lem19}.\begin{proof}It is enough to prove of the convergence of the characteristic functions\begin{align}\label{Eelf}
\lim_{n\to+\infty}\mathbb{E}[e^{i\lambda\langle\xi_n,f\rangle}]=e^{-\lambda^2\sigma^2},\ \forall\ \lambda\in\R.
\end{align} Given a function $f\in W^{1,p}(S^1)$, we will prove that $f_{N}:=f*K_{N}$
 approximates $f$ in $W^{1,p}(S^1)$,  where $f_1*f_2(x):=\int_0^{2\pi}f_1(y)f_2(x-y)dy,$ and
$K_{N}(x)$ is the F\'ejer kernel\begin{align*} &K_{N}(x)=\frac{1}{2\pi}\sum_{j=-N}^N\left(1-\frac{|j|}{N}\right)e^{ij x}=\frac{N}{2\pi}\left(\frac{\sin(N x/2)}{N\sin( x/2)}\right)^2,\,\,\, N>0,\ N\in\Z.\end{align*}In fact $K_N(x)\geq 0,\ \|K_N\|_{L^1}=1,$ $f_{N}'=f'*K_{N}$, and\begin{align*} &\lim_{N\to+\infty}\|K_{N}\|_{L^1(\delta,2\pi-\delta)}=0,\ \forall\ \delta\in (0,\pi).\end{align*} The following results are classical\begin{align*} &\|g*K_{N}\|_{L^p}\leq \|g\|_{L^p},\ \lim_{N\to+\infty}\|g*K_{N}-g\|_{L^p}=0,\ \forall\ g\in L^p(0,2\pi).\end{align*}Thus we have \begin{align}\label{fN} &\|f'_{N}\|_{L^p}\leq \|f'\|_{L^p},\ \lim_{N\to+\infty}\|f'_{N}-f'\|_{L^p}=0.\end{align}We also have\begin{align*} &f_{N}(x)=\sum_{j=-N}^N\left(1-\frac{|j|}{N}\right)a_je^{ij x},\end{align*} where $a_j$ is defined in Theorem \ref{lem19}. Since $f$ is real valued and $\int_0^{2\pi}f(x)dx=0$, we have $a_0=0,\ a_{-j}=\overline{a_j}$, $f_{N}(x)=g_N(e^{ix})+\overline{g_N(e^{ix})} $ with $g_{N}(z)=\sum_{j=1}^N\left(1-j/N\right)a_jz^j. $ By Lemma \ref{lem21}, $\langle\xi_n,f_N\rangle $ converges in law to $J_N\sim N(0, 2\sigma_N^2)$ as $n\to+\infty$ for every fixed $N$, where\begin{align*}
\sigma_N^2&=\frac{2}{\beta}\sum_{j=1}^{N}j\left(1-j/N\right)^2|a_j|^2
\end{align*}with $\sigma_{N+1}\geq \sigma_N$. Thus \begin{align}\label{Eel}
\lim_{n\to+\infty}\mathbb{E}[e^{i\lambda\langle\xi_n,f_N\rangle}]=e^{-\lambda^2\sigma_N^2},\ \forall\ \lambda\in\R.
\end{align}As $a_0=0,\ \int_0^{2\pi}f_N(x)dx=0,$ by Lemma \ref{lem18} and  Fatou's Lemma we have $$2\sigma_N^2=\mathbb E[J_N^2]\leq\liminf_{n\to+\infty}\mathbb{E}[|\langle\xi_n,f_N\rangle|^2]\leq C \|f'_{N}\|_{L^p}^2,$$ which implies 

$$\sigma_N^2\leq C\|f'_{N}\|_{L^p}^2\leq C\|f'\|_{L^p}^2 .$$ Thus by monotone convergence theorem we have\begin{align}\label{sN}
\sigma^2=\lim_{N\to+\infty}\sigma_N^2\leq C\|f'\|_{L^p}^2<+\infty,
\end{align}where $\sigma$ is defined in Theorem \ref{lem19}. By Lemma \ref{lem18} again we have $$\mathbb{E}[|\langle\xi_n,f_N\rangle-\langle\xi_n,f\rangle|^2]\leq C \|f'_{N}-f'\|_{L^p}^2,$$ and thus\begin{align*}
&|\mathbb{E}[e^{i\lambda\langle\xi_n,f_N\rangle}-e^{i\lambda\langle\xi_n,f\rangle}]|\leq \mathbb{E}[|\lambda||\langle\xi_n,f_N\rangle-\langle\xi_n,f\rangle|]\leq C|\lambda| \|f'_{N}-f'\|_{L^p},\ \forall\ \lambda\in\R,
\end{align*}which together with \eqref{Eel}  gives \begin{align}\label{Eef}
&\limsup_{n\to+\infty}|\mathbb{E}[e^{i\lambda\langle\xi_n,f\rangle}]-e^{-\lambda^2\sigma^2}|\leq C|\lambda| \|f'_{N}-f'\|_{L^p}+|e^{-\lambda^2\sigma_N^2}-e^{-\lambda^2\sigma^2}|
\end{align}for every $\lambda\in\R,\ N>0,\ N\in\Z$. By \eqref{fN}, \eqref{sN}, \eqref{Eef} we have \begin{align*}
&\limsup_{n\to+\infty}|\mathbb{E}[e^{i\lambda\langle\xi_n,f\rangle}]-e^{-\lambda^2\sigma^2}|\\ \leq& \limsup_{N\to+\infty}(C|\lambda| \|f'_{N}-f'\|_{L^p}+|e^{-\lambda^2\sigma_N^2}-e^{-\lambda^2\sigma^2}|)=0,\ \forall\ \lambda\in\R,
\end{align*}which implies \eqref{Eelf}. This completes the proof.\end{proof}

\end{document}